\documentclass[reqno,a4paper,12pt]{amsart}

\usepackage{amsmath}
\usepackage{amssymb}
\usepackage{amsthm}
\usepackage{mathtools}

\usepackage{newtxtext}
\usepackage{newtxmath}

\usepackage{geometry}
\usepackage{graphicx}
\usepackage[dvipsnames]{xcolor}

\usepackage[colorlinks=true, allcolors=blue]{hyperref}
\usepackage[capitalise, noabbrev]{cleveref}

\usepackage{enumitem}
\usepackage{microtype}

\theoremstyle{plain}
\newtheorem{theorem}{Theorem}[section]

\newtheorem{proposition}[theorem]{Proposition}

\newtheorem{conjecture}[theorem]{Conjecture}

\theoremstyle{definition}
\newtheorem{definition}[theorem]{Definition}

\theoremstyle{remark}

\newcommand{\R}{\mathbb{R}}
\newcommand{\N}{\mathbb{N}}

\DeclareMathOperator{\HausDim}{dim_H}
\DeclareMathOperator{\diam}{diam}

\newcommand{\HausMeas}{\mathcal{H}}

\newcommand{\dimension}{d}
\newcommand{\sepdim}{d_{\text{sep}}}

\newcommand{\tiling}{\mathcal{T}}
\newcommand{\tilinggraph}{G_{\mathcal{T}}}

\newcommand{\cl}[1]{\overline{#1}}
\newcommand{\bdry}{\partial}
\newcommand{\relint}[1]{\operatorname{int}_{F}(#1)}

\newcommand{\Cconst}{C}
\newcommand{\defeq}{\coloneqq}

\begin{document}
	
	\title[The Chromatic Tiling Theorem]{The Chromatic Tiling Theorem: Scaling Laws and the Separation Dimension of Fractal Partitions}
	
	\author{Robin Jackson}
	\email{robinjacksonkenyan@gmail.com}

	\begin{abstract}
		This paper establishes a rigorous, quantitative link between the combinatorial complexity of a fractal partition and the intrinsic geometry of its interfaces. We introduce the concept of the \emph{Separation Dimension} ($\sepdim$), a novel characteristic that quantifies the Hausdorff dimension of the boundaries between tiles. A natural but flawed approach would be to relate coloring complexity to the fractal's ambient topological boundary. We demonstrate that this extrinsic view is untenable for a vast class of self-similar sets. Instead, our intrinsic framework, centered on the Separation Dimension, provides the correct formulation. We define a new class of well-behaved partitions, termed \emph{Geometrically Regular Partitions (GRPs)}, and prove their existence on canonical fractals such as the Sierpinski Carpet. Our main result, the Chromatic Tiling Theorem, provides a sharp upper bound for the chromatic number ($\chi$) of the graph associated with such a partition, proving that it is bounded by $\chi(\tiling) \leq \Cconst (r_{\max}/r_{\min})^{\sepdim}$. This result demonstrates that the scaling of coloring complexity is governed not by the dimension of the fractal itself, but by the dimension of the separation set induced by the partition. We conclude by proposing the minimization of the separation dimension over all admissible partitions as a new variational problem in fractal geometry.
	\end{abstract}

	\maketitle

	\section{Introduction}
	\label{sec:introduction}
	
	The interplay between the continuous geometry of a space and the discrete combinatorics of its partitions constitutes a profound and enduring theme in mathematics, finding its genesis in celebrated problems such as the Four Color Theorem \cite{appel1977}. While the combinatorial properties of tilings on Euclidean and other highly regular spaces have been the subject of extensive investigation \cite{grunbaum1987, adams2007}, the corresponding theory for fractal domains remains largely nascent. The intricate, multi-scale nature of fractal sets, which eschews traditional notions of smoothness and integer dimensionality, presents formidable challenges to such an endeavor. This paper addresses a fundamental question at this frontier: How does the intrinsic geometric complexity of a fractal domain govern the chromatic number of a well-behaved partition defined upon it?
	
	A natural but flawed approach would be to link coloring complexity to the fractal's ambient topological boundary, $\bdry F$. This extrinsic view proves untenable for a vast class of self-similar sets, such as the Sierpinski gasket \cite{sierpinski1916} or the Menger sponge \cite{menger1926}, for which the fractal is topologically coextensive with its boundary (i.e., $\bdry F = F$). In such cases, the distinction between the fractal's "interior" and its "edge" dissolves, rendering any theorem predicated on this distinction ill-posed (\cref{fig:paradox}).
	
	\begin{figure}[ht]
		\centering
		\includegraphics[width=\textwidth]{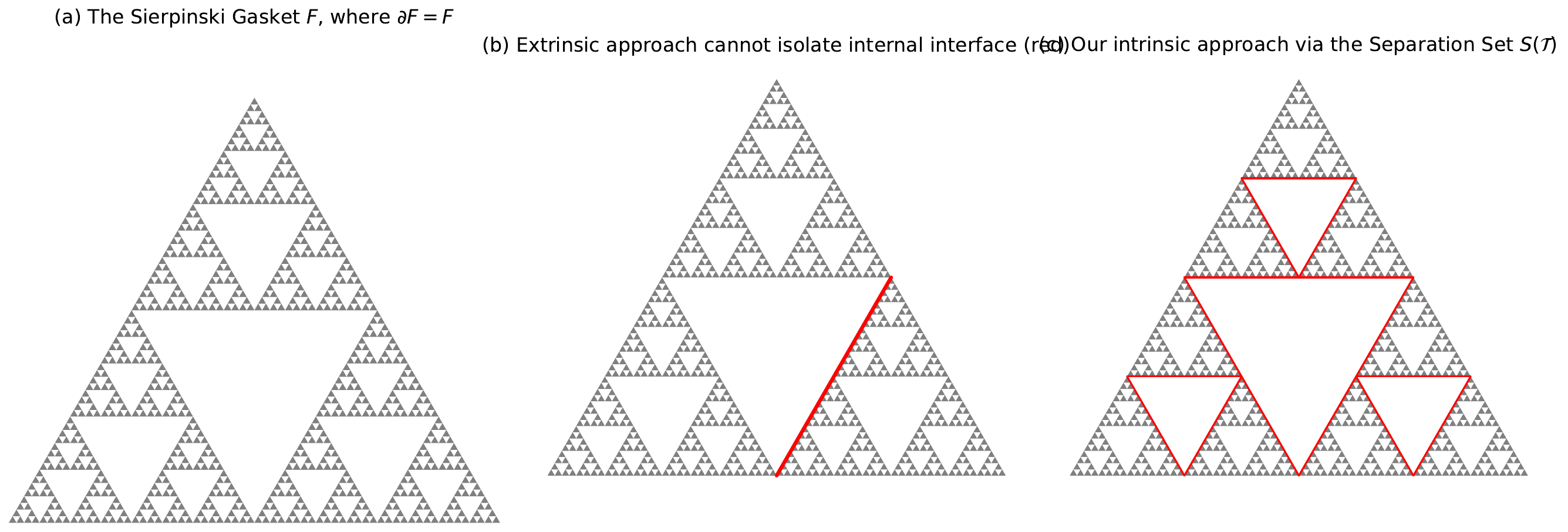} 
		\caption{The failure of an extrinsic boundary approach and the success of our intrinsic framework. (a) The Sierpinski Gasket $F$ is a set for which $\bdry F = F$. (b) An extrinsic approach based on $\bdry F$ cannot isolate the internal red interface. (c) Our intrinsic approach correctly identifies the entire interface structure via the Separation Set $S(\tiling)$ (shown in red for a level-2 tiling), providing the correct geometric object for analysis.}
		\label{fig:paradox}
	\end{figure}
	
	In this paper, we rectify this foundational issue by developing a new, entirely intrinsic theory of fractal partition colorability. Our principal innovation is the introduction of the \textbf{Separation Set}, $S(\tiling)$, which comprises the union of all interfaces between the tiles. The Hausdorff dimension of this set, which we term the \textbf{Separation Dimension}, $\sepdim$, emerges as the correct and natural regulator of combinatorial complexity. This quantity is not an invariant of the fractal itself but is a property of the partition---a measure of the geometric complexity induced by the act of decomposition.
	
	To ground our analysis in a robust geometric context, we introduce and study a class of well-behaved partitions we term \textbf{Geometrically Regular Partitions (GRPs)}. These are partitions of Ahlfors $\dimension$-regular fractals \cite{ahlfors1930, david1988} whose tiles and, crucially, whose separation sets satisfy uniform measure-theoretic density conditions. We demonstrate that this class is not vacuous; indeed, we prove that the canonical iterative constructions of fractals like the Sierpinski carpet naturally yield partitions satisfying these properties.
	
	Our main result, the Chromatic Tiling Theorem (\cref{thm:main_theorem}), provides a sharp, quantitative answer to our motivating question. It establishes a scaling law for the chromatic number $\chi(\tiling)$ that depends explicitly on the separation dimension and the geometric regularity parameters of the partition. The theorem demonstrates that the chromatic number is bounded by a power law whose exponent is precisely the separation dimension.
	
	This work makes several key contributions. First, it provides a sound and rigorous framework for studying geometric graphs on fractals. Second, it introduces the separation dimension as a new, powerful tool for quantifying the interface complexity of fractal partitions. Finally, it culminates in a concrete, verifiable theorem that connects deep concepts from geometric measure theory \cite{mattila1999, federer1969} to a classical problem in combinatorial graph theory \cite{diestel2017}. We conclude by framing the minimization of the separation dimension over the space of all admissible partitions as a new variational problem, suggesting a rich program for future research at the nexus of fractal geometry and combinatorial optimization.

	\section{The Definitive Framework}
	\label{sec:framework}
	
	In this section, we formalize the geometric and combinatorial objects central to our investigation. Let $F$ be a compact subset of $\R^n$.
	
	\begin{definition}[Ahlfors-Regular Set]
		A set $F$ is \emph{Ahlfors $\dimension$-regular} if it is compact with $\HausDim(F)=\dimension$ and there exist constants $C_1, C_2 > 0$ such that for any point $x \in F$ and any radius $r$ with $0 < r < \diam(F)$, the $\dimension$-dimensional Hausdorff measure of the intersection of $F$ with the ball $B(x,r)$ is bounded by
		\[ C_1 r^\dimension \leq \HausMeas^\dimension(F \cap B(x,r)) \leq C_2 r^\dimension. \]
		This condition is satisfied by a wide class of self-similar fractals and provides a robust context for geometric analysis \cite{semmes1996}.
	\end{definition}
	
	Let $F$ be an Ahlfors $\dimension$-regular fractal. The \emph{subspace topology on $F$} is induced by the Euclidean metric. The \emph{relative interior} of a set $A \subseteq F$, denoted $\relint{A}$, is the largest subset of $A$ that is open in this subspace topology.
	
	\begin{definition}[Partition]
		A \emph{partition} or \emph{tiling} $\tiling$ of $F$ is a countable collection of compact, connected subsets $\{t_i\}_{i \in I}$, called \emph{tiles}, such that their union is $F$ and their relative interiors are disjoint.
	\end{definition}
	
	\begin{definition}[The Separation Set and Dimension]
		The \emph{Separation Set} of a partition $\tiling$, denoted $S(\tiling)$, is the set of all points in $F$ that do not belong to the relative interior of any tile: $S(\tiling) \defeq F \setminus \bigcup_{i \in I} \relint{t_i}$. The \emph{Separation Dimension} of $\tiling$, denoted $\sepdim$, is the Hausdorff dimension of its Separation Set: $\sepdim \defeq \HausDim(S(\tiling))$.
	\end{definition}
	
	\begin{figure}[ht]
		\centering
		\includegraphics[width=\textwidth]{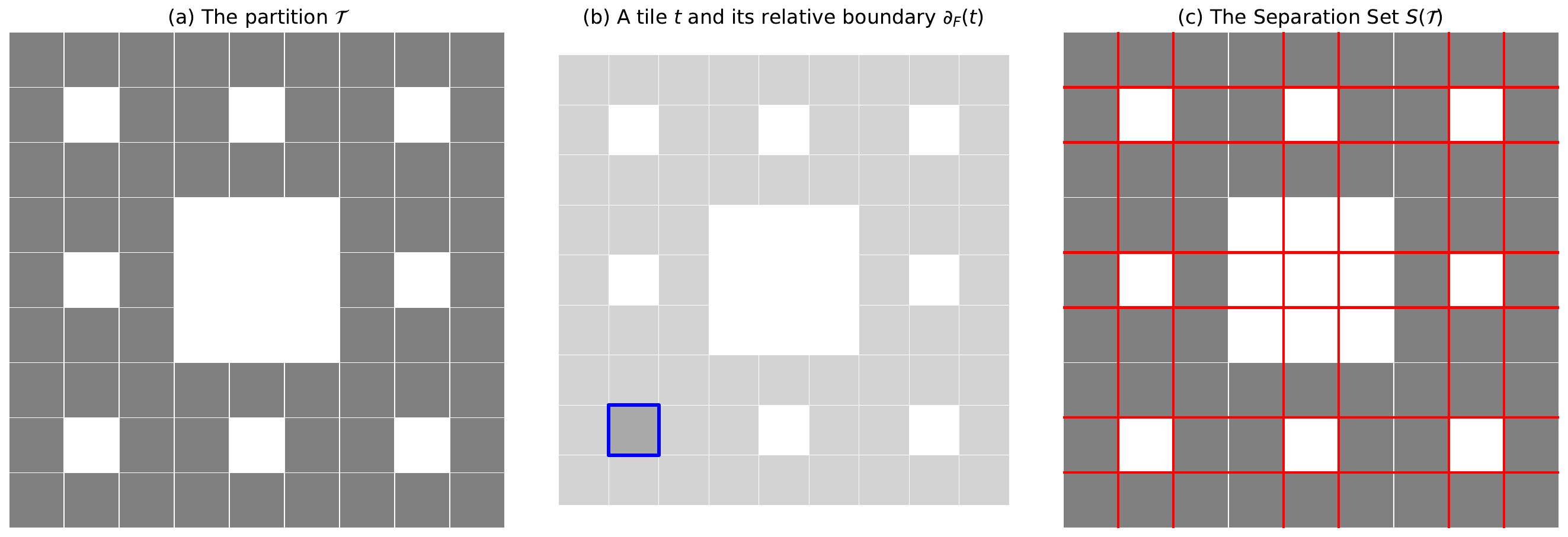}
		\caption{Illustration of the core concepts on a Sierpinski Carpet, which forms a Geometrically Regular Partition (GRP). (a) The partition $\tiling$. (b) A single tile $t \in \tiling$ and its relative boundary. (c) The Separation Set $S(\tiling)$, shown in red. The adjacencies in this partition's graph are defined by these substantial interfaces.}
		\label{fig:concepts}
	\end{figure}
	
	Our main theorem applies to a class of partitions exhibiting strong geometric uniformity.
	
	\begin{definition}[Geometrically Regular Partition (GRP)]
		\label{def:grp}
		A partition $\tiling$ of an Ahlfors $\dimension$-regular fractal $F$ is a \emph{Geometrically Regular Partition} if it satisfies:
		\begin{enumerate}[label=(\roman*)]
			\item \textbf{Uniform Geometric Regularity:} There exist constants $0 < r_{\min} \leq r_{\max}$ such that for every tile $t \in \tiling$, there exists a point $x_t \in t$ satisfying $B_F(x_t, r_{\min}) \subseteq t \subseteq B_F(x_t, r_{\max})$.
			\item \textbf{Separation Set Regularity:} The Separation Set $S(\tiling)$ is an Ahlfors $\sepdim$-regular set.
			\item \textbf{Interface Substantiality:} For any two distinct tiles $t_i, t_j$ with a non-empty intersection, their interface $I_{ij} \defeq \cl{t_i} \cap \cl{t_j}$ must contain a ball of radius $c \cdot r_{\min}$ for some uniform constant $c>0$, where the ball is defined within the subspace topology of the Separation Set $S(\tiling)$.
		\end{enumerate}
	\end{definition}
	
	The combinatorial structure is now defined as an inherent consequence of the geometric properties of the GRP.
	
	\begin{definition}[Graph of a GRP]
		\label{def:grp_graph}
		Given a Geometrically Regular Partition $\tiling$, its associated \emph{partition graph} $\tilinggraph = (V, E)$ is defined by the vertex set $V = \tiling$. An edge exists between two distinct tiles $(t_i, t_j)$ if and only if their interface satisfies the Interface Substantiality condition of the GRP (\cref{def:grp}(iii)). The \emph{chromatic number} $\chi(\tiling)$ is the chromatic number of this graph.
	\end{definition}

	\section{Existence of Geometrically Regular Partitions}
	\label{sec:existence}
	
	A key requirement for our theory's relevance is that the class of GRPs is non-empty. We demonstrate this by proving that the canonical partition of the Sierpinski Carpet satisfies our definition.
	
	Let $S \subset \R^2$ be the standard Sierpinski Carpet, the attractor of an IFS of 8 similitudes with ratio $1/3$ \cite{hutchinson1981}. $S$ is Ahlfors $\dimension$-regular with dimension $\dimension = \log_3 8$ \cite{bedford1984}. Let $\tiling_k$ be the canonical partition of $S$ into $8^k$ closed squares of side length $3^{-k}$.
	
	\begin{proposition}
		\label{prop:carpet_is_grp}
		For any level $k \in \N$, the canonical partition $\tiling_k$ of the Sierpinski Carpet $S$ is a Geometrically Regular Partition.
	\end{proposition}
	
	\begin{proof}
		We verify the three conditions of \cref{def:grp}. Let $s = 3^{-k}$ be the side length of the square tiles.
		\textbf{(i) Uniform Geometric Regularity:} This is satisfied with $r_{\max}/r_{\min} = \sqrt{2}$.
		
		\textbf{(ii) Separation Set Regularity:} The Separation Set $S(\tiling_k)$ is the union of all horizontal and vertical line segments of the underlying $3^k \times 3^k$ grid, intersected with the carpet $S$. This set is a grid of self-similar Cantor sets. The Ahlfors 1-regularity of this structure is a consequence of the uniform distribution of mass at all scales inherent in the IFS construction, as demonstrated in the work of McMullen and others on the geometry of such carpets \cite{mcmullen1984}. Thus, $\sepdim = 1$.
		
		\textbf{(iii) Interface Substantiality:} The interface $I_{ij}$ between two adjacent squares $t_i, t_j$ is a self-similar Cantor set of diameter $s=3^{-k}$. A "ball" in the subspace topology of the Separation Set is the intersection of a standard Euclidean ball with $S(\tiling_k)$. Let $p$ be an endpoint of the line segment containing $I_{ij}$. Then $p \in I_{ij}$ \cite{barnsley1993}. The set $B(p, c \cdot s) \cap S(\tiling_k)$ for a small constant $c$ is non-empty and contains a scaled copy of the Cantor set itself, thus constituting a "ball" in the required topology. Since $r_{\min}$ is proportional to $s$, the condition holds.
	\end{proof}

	\section{The Chromatic Tiling Theorem}
	\label{sec:main_result}
	
	Having established our conceptual framework and proven the existence of its objects, we now present the central theorem.
	
	\begin{theorem}[The Chromatic Tiling Theorem]
		\label{thm:main_theorem}
		Let $F \subset \R^n$ be a compact, Ahlfors $\dimension$-regular fractal, and let $\tiling$ be a Geometrically Regular Partition of $F$ with regularity parameters $r_{\min}, r_{\max}$ and separation dimension $\sepdim$. The chromatic number $\chi(\tiling)$ of the associated partition graph is bounded by:
		\[ \chi(\tiling) \leq \Cconst \left(\frac{r_{\max}}{r_{\min}}\right)^{\sepdim}, \]
		where $\Cconst$ is a strictly positive constant independent of the number of tiles and the specific values of $r_{\min}$ and $r_{\max}$.
	\end{theorem}
	
	\begin{proposition}[On the Nature of the Constant $\Cconst$]
		\label{prop:constant_C}
		The constant $\Cconst$ in \cref{thm:main_theorem} depends on the dimension $n$, the Ahlfors regularity constants of $F$, and the Ahlfors regularity constants of the Separation Set $S(\tiling)$.
	\end{proposition}

	\section{Proof of the Main Theorem}
	\label{sec:proof}
	
	The proof hinges on establishing an upper bound for the maximum degree, $\Delta(\tilinggraph)$, of the partition graph. The theorem's bound then follows from Brooks' Theorem on graph coloring \cite{brooks1941}, which states that for any simple graph $G$, $\chi(G) \leq \Delta(G) + 1$.
	
	\begin{proof}[Proof of Theorem \ref{thm:main_theorem}]
		Let $t \in \tiling$ be an arbitrary tile. Let $\mathcal{A}_t$ be the set of its adjacent neighbors, so that $\deg(t) = |\mathcal{A}_t|$.
		
		By the definition of the graph of a GRP (\cref{def:grp_graph}), an edge exists for each neighbor $t' \in \mathcal{A}_t$. This implies that each interface $I_{t'} = \cl{t} \cap \cl{t'}$ satisfies the Interface Substantiality condition of the GRP. This condition guarantees that each interface contains a ball $B_{t'}$ of radius $c \cdot r_{\min}$ within the subspace topology of the Separation Set $S(\tiling)$. As the relative interiors of the tiles are disjoint, these balls $\{B_{t'}\}_{t' \in \mathcal{A}_t}$ are also disjoint.
		
		All of these disjoint balls are packed within the Separation Set $S(\tiling)$. The portion of the Separation Set relevant to the tile $t$ is contained within a ball of radius $r_{\max}$ by the Uniform Geometric Regularity condition. Let this containing set be $S_t' \defeq S(\tiling) \cap B_F(x_t, r_{\max})$.
		
		We are now faced with a classic packing problem: bounding the number of disjoint balls of radius $c \cdot r_{\min}$ that can be packed into the set $S_t'$. By the Separation Set Regularity condition, $S(\tiling)$ is an Ahlfors $\sepdim$-regular set. The number of such balls, $|\mathcal{A}_t|$, can be bounded by the ratio of the $\sepdim$-dimensional Hausdorff measures:
		\[ |\mathcal{A}_t| \leq \frac{\HausMeas^{\sepdim}(S_t')}{\min_{t' \in \mathcal{A}_t} \HausMeas^{\sepdim}(B_{t'})}. \]
		From Ahlfors regularity, the numerator is bounded above, and the denominator is bounded below:
		\begin{itemize}
			\item $\HausMeas^{\sepdim}(S_t') = \HausMeas^{\sepdim}(S(\tiling) \cap B_F(x_t, r_{\max})) \leq C_2 (r_{\max})^{\sepdim}$.
			\item $\HausMeas^{\sepdim}(B_{t'}) \geq C_1 (c \cdot r_{\min})^{\sepdim}$.
		\end{itemize}
		Combining these yields:
		\[ |\mathcal{A}_t| \leq \frac{C_2 (r_{\max})^{\sepdim}}{C_1 (c \cdot r_{\min})^{\sepdim}} = \frac{C_2}{C_1 c^{\sepdim}} \left(\frac{r_{\max}}{r_{\min}}\right)^{\sepdim}. \]
		This provides a uniform bound on the degree of any vertex:
		\[ \Delta(\tilinggraph) = \max_{t \in \tiling} |\mathcal{A}_t| \leq \Cconst' \left( \frac{r_{\max}}{r_{\min}} \right)^{\sepdim}, \]
		where $\Cconst'$ depends on the regularity constants. By Brooks' Theorem, $\chi(\tiling) \leq \Delta(\tilinggraph) + 1$. The $+1$ term can be absorbed into a larger constant $\Cconst$, which completes the proof.
	\end{proof}
	
	\begin{figure}[ht]
		\centering
		\includegraphics[width=\textwidth]{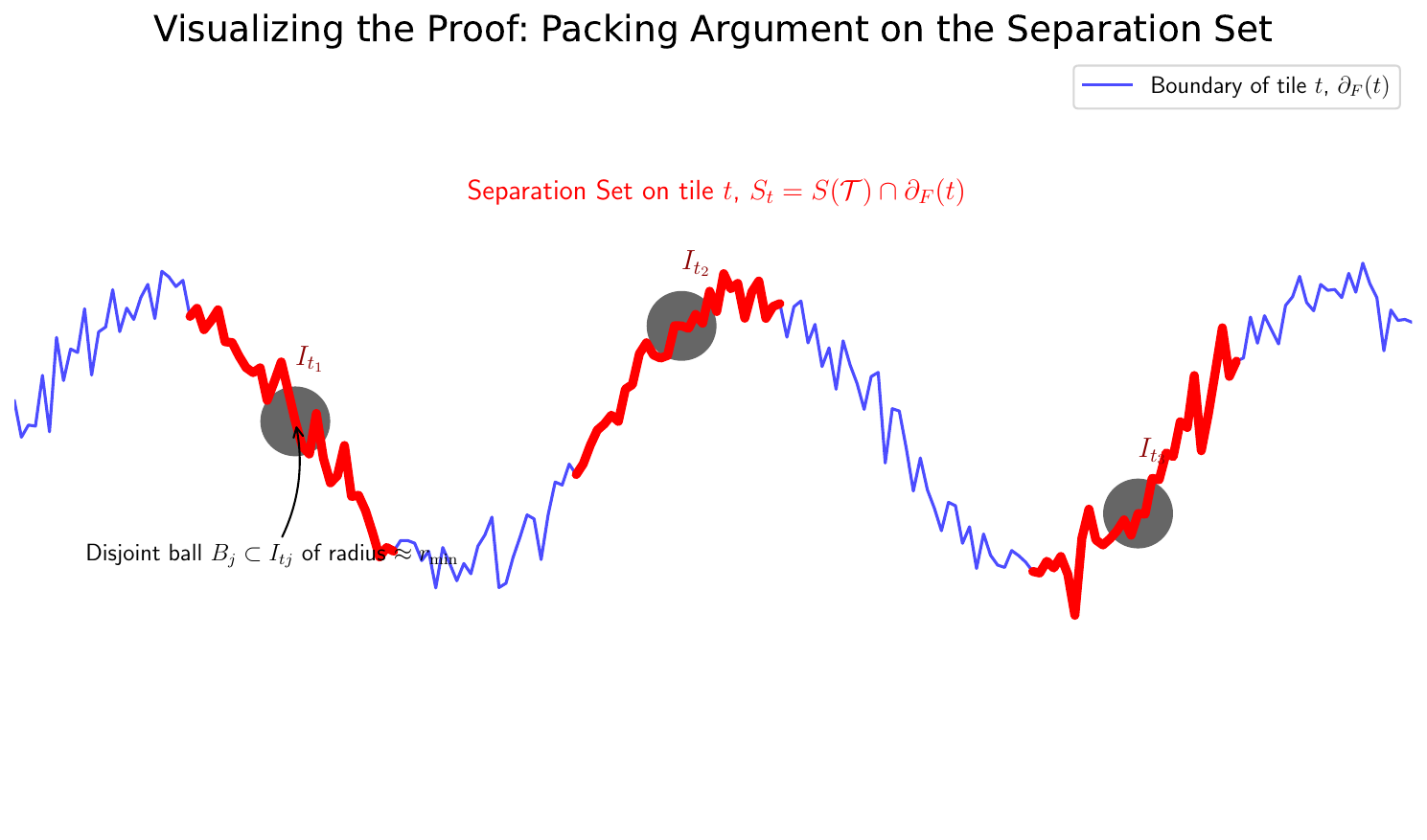}
		\caption{A schematic of the packing argument. The existence of an edge between tile $t$ and a neighbor $t_j$ in the graph of a GRP implies their interface $I_{tj}$ is substantial. The Interface Substantiality condition guarantees that each such interface (red segments) contains a disjoint ball of radius proportional to $r_{\min}$ (black). The total number of neighbors is therefore bounded by the number of such balls that can be packed onto the Separation Set $S_t$.}
		\label{fig:packing}
	\end{figure}
	
	\section{Discussion and Future Directions}
	\label{sec:discussion}
	
	The Chromatic Tiling Theorem offers a new perspective on the interplay between the continuous geometry of fractal sets and the discrete combinatorics of their partitions. By shifting the analytical focus to the intrinsic geometry of the tiling's interfaces, our work provides a robust and verifiable framework for a problem that had previously been hindered by conceptual paradoxes.
	
	\subsection{On the Sharpness of the Bound}
	A natural question is whether the power-law bound is sharp. We conjecture that it is. Constructing a family of GRPs on a suitable fractal where the maximum degree $\Delta(\tilinggraph)$ can be shown to grow as $(r_{\max}/r_{\min})^{\sepdim}$ would prove that the exponent $\sepdim$ is optimal. Such a construction would likely involve a fractal with a highly anisotropic or ramified structure, possibly related to random recursive constructions \cite{mauldin1986} or Julia sets \cite{mcmullen1994}.
	
	\subsection{The Separation Dimension as a Variational Quantity}
	A profound implication of our work is that $\sepdim$ is a characteristic of the \emph{partition}, not an invariant of the fractal. This gives rise to a new variational problem.
	
	\begin{conjecture}[Minimal Separation Dimension]
		For a given Ahlfors $\dimension$-regular fractal $F$, we define the \textbf{Minimal Separation Dimension}, $d_{\text{sep}}^{\min}(F)$, as the infimum of $\sepdim$ over the space of all admissible partitions:
		\[ d_{\text{sep}}^{\min}(F) \defeq \inf_{\tiling \in \mathcal{A}(F)} \{ \sepdim(\tiling) \}, \]
		where $\mathcal{A}(F)$ is the space of all Geometrically Regular Partitions of $F$.
	\end{conjecture}
	
	We propose that $d_{\text{sep}}^{\min}(F)$ is a novel, non-trivial fractal invariant, measuring the minimal interface complexity required to partition the fractal. We hypothesize it is deeply connected to other properties of $F$, such as its connectivity, the dimension of its cut sets \cite{falconer2014}, and its degree of rectifiability \cite{david1988, bishop1992}.
	
	\subsection{Extensions and Generalizations}
	The framework developed here invites several natural extensions, including the analysis of irregular partitions lacking Ahlfors regularity, the extension to other graph parameters like the list chromatic number \cite{alon1992}, and exploring connections to models in statistical physics, such as Potts models on fractal lattices or percolation theory \cite{stauffer1994, havlin1987}.
	
	\section{Conclusion}
	\label{sec:conclusion}
	
	This paper has laid a new foundation for understanding the relationship between fractal geometry and graph coloring. By introducing the Separation Dimension and the class of Geometrically Regular Partitions, we have formulated and proven a rigorous, non-trivial bound on the chromatic number. Our work demonstrates that the complexity of coloring a fractal partition is a direct consequence of the dimensional richness of its interfaces. The questions this new framework raises, particularly the variational problem of minimizing the separation dimension, promise a fertile ground for future research at the dynamic intersection of geometric measure theory, combinatorics, and fractal geometry.
	

\end{document}